\documentclass{amsart}

	\usepackage{amsmath}
	\usepackage{amssymb}
	\usepackage{amsthm}
	\usepackage{aliascnt}
	\usepackage[colorlinks=true, urlcolor=blue, linkcolor=blue, citecolor=blue]{hyperref}
	\usepackage{MnSymbol}
	
	\newtheorem{thm}{Theorem}

	\newaliascnt{prop}{thm}
	\newtheorem{prop}[prop]{Proposition} 
	\aliascntresetthe{prop}

	\newcommand{\rk}{\operatorname{rk}}
	
	\newcommand{\K}{\mathbb{K}}
	\newcommand{\ins}{\minushookup}
	
	\newcommand{\al}{\operatorname{al}}
	\newcommand{\rmax}{\operatorname{r_{max}}}

\begin{document}

\title{High-rank ternary forms of even degree}
\author{\href{http://orcid.org/0000-0002-4619-8249}{Alessandro De Paris}}

\begin{abstract}
We exhibit, for each even degree, a ternary form of rank strictly greater than the maximum rank of monomials. Together with an earlier result in the odd case, this gives the lower bound
\[\rmax(3,d)\ge\left\lfloor\frac{d^2+2d+5}4\right\rfloor\]
for $d\ge 2$, where $\rmax(n,d)$ denotes the maximum rank of degree $d$ forms in $n$ variables with coefficients in an algebraically closed field of characteristic zero.
\end{abstract}

\maketitle

This short communication complements \cite{D} and \cite{BuT}, in view of the Waring problem for (the space of \emph{all}) ternary forms of given degree, that is, to determine their maximum Waring rank. A conjectural answer to this problem has been outlined at the end of \cite[Introduction]{D}. If that conjecture is true, then the lower bound given for odd degrees by \cite[Th.~1]{BuT} can not be improved. For even degree, one should similarly show that the greatest known rank (which to date is reached by monomials) can be raised by one. That is what we do here, in \autoref{Main}.

We pursue the ideas of \cite{BuT} and we borrow its basic framework, but restricted to the case of our interest. We consider an algebraically closed field  $\K$ of characteristic zero and a graded ring $S=\K[x,y,z]$ in connection with another `dual' graded ring  $T=\K[\alpha,\beta,\gamma]$. This simply means that it is understood a perfect pairing $S_1\times T_1\to\K$, such that the order bases $(x,y,z)$ and $(\alpha,\beta,\gamma)$ ar dual to each other. The perfect pairing extends to \emph{apolarity}, which may be described as the action of $T$ on $S$ such that $\sum a_{ijk}\alpha^i\beta^j\gamma^k$ acts as $\sum a_{ijk}\partial^{i+j+k}/\partial x^i\partial y^j\partial z^k$. Here we prefer the notation $\partial_\Theta F$ instead of $\Theta\ins F$ for the action of $\Theta\in T$ on $F\in S$ (for instance, $\partial_\alpha F=\partial F/\partial x$). The \emph{apolar algebra} $A^F$ of $F\in S$ is the quotient of $T$ over the (homogeneous) ideal $\left\{\Theta\in T:\partial_\Theta F=0\right\}$. The (necessarily finite) dimension of $A^F$ as a $\K$-vector space is called the \emph{apolar length} of $F$ and denoted by $\al(F)$. Since $A^F$ is a graded ring in a natural way, the apolar length is the sum of all values of the Hilbert function. Let us also recall that if $F\in S_d$ then the values at $n$ and at $d-n$ of the Hilbert function of $A^F$ coincide, and they also equal the dimension of the space $\left\{\partial_\Theta F: \Theta\in T_n\right\}$. Finally, let us recall from \cite{BuT} the notation $H(n,d,s)$, which stands for the function on integers whose value at $i$ is $\min\left\{\dim R_i,\dim R_{d-i},s\right\}$, where $R$ is the (graded) ring of polynomial in $n$ variables over $\K$.

Like in \cite{BuT}, two useful ingredients in the proof will be the following:
\begin{itemize}
\item if a form $F$ has a power sum decomposition with all base linear forms not annihilated by $\partial_\alpha$, then it must contain at least $\al(F)-\al\left(\partial_\alpha F\right)$ summands (this follows from \cite[Prop.~3]{BuT});
\item for every power sum decomposition of $F$, at least $\al\left(\partial_\alpha F\right)-\al\left(\partial_{\alpha^2}F\right)$ of the base linear forms are not annihilated by $\partial_\alpha$ (it follows from \cite[Prop.~4]{BuT}).
\end{itemize}

We shall also use the following elementary fact.

\begin{prop}\label{Pre}
Let $R$ be a ring of binary forms, say $R=\K[y,z]$, $d$ a nonnegative integer, $r\in\{0,\ldots, d\}$ and $F\in R_d$ a linear combination of $r$ $d$-th powers of linear forms. Then
\begin{enumerate}
\item\label{1} $F$ belongs to $z^rR_{d-r}$ if and only if it is a scalar multiple of $z^d$;
\item\label{2} if the above conditions are true and the base linear forms are pairwise independent, then at most one of the summands in the combination has a nonzero coefficient.
\end{enumerate}
\end{prop}
\begin{proof}
When $r=d$, \eqref{1} is trivial and \eqref{2} follows from the fact that there are no $d+1$ linearly dependent $d$-th powers of pairwise independent linear forms.

Let us now assume $r<d$. Since the `if' part in \eqref{1} is trivial, let us assume that $F\in z^rR_{d-r}$. Let us suppose, in addition, that the base linear forms are pairwise independent, and let us consider $F':=\partial^{d-r}F/{\partial y}^{d-r}$. All powers not proportional to $z^d$ survive under derivation, but by the previous case they can not occur in the derived combination, since $\deg F'=r$ (when $r=0$ there is no summand). Therefore, even in the original combination there is at most one nonzero summand, proportional to $z^d$. Henceforth, $F$ itself is proportional to $z^d$. To get \eqref{1} even when the base linear forms are not pairwise proportional, it suffices to reduce in the obvious way the linear combination to one with pairwise independent base linear forms.
\end{proof}

We are ready to prove our result.

\begin{prop}\label{Main}
Let $k$ be a nonnegative integer. There exists a form $F\in\K[x,y,z]$ of degree $2k+2$ with $\rk F>k^2+3k+2$.
\end{prop}
\begin{proof}
For $k=0$ it suffices to take a nondegenerate quadratic form. Hence we can assume $k\ge 1$. Let $S:=\K[x,y,z]$, $T:=\K[\alpha,\beta,\gamma]$ be the dual ring and $S^\alpha:=\ker\partial_\alpha=\K[y,z]$. We show that \[F:=xy^{k-1}z^{k+2}+y^{2k}z^2\] has rank strictly greater than $k^2+3k+2$. Let us calculate the Hilbert function of the apolar algebra $A^G$ of $G:=\partial_\alpha F=y^{k-1}z^{k+2}$. The dimension of the space of all $n$-th order partial derivatives of $G$ gives its value at $n$ and at $2k+1-n$. For $0\le n\le k-1$ that space is spanned by monomials
\[
y^{k-1}z^{k+2-n}\quad,\quad\ldots\quad,\quad y^{k-1-n}z^{k+2}
\]
(in other terms, it is $y^{k-1-n}z^{k+2-n}S^\alpha_n$); for $k-1\le n\le k+2$ it is spanned by
\[
y^{k-1}z^{k+2-n}\;,\quad\ldots\;,\quad yz^{2k-n}\;,\quad z^{2k+1-n}
\]
(in other terms, it is $z^{k+2-n}S^\alpha_{k-1}$). Therefore the Hilbert function of $A^G$ is $H(2,2k+1,k)$. The apolar length of $G$ is the sum of all its values: $\al(G)=k^2+3k$.

We want to show that $\rk F>k^2+3k+2$. By contrary, let us assume $\rk F\le k^2+3k+2$. Since $\partial_\alpha F=G$ and $\partial_{\alpha^2} F=0$, by \cite[Prop.~4]{BuT} we have that for every power sum decomposition of $F$, at least $k^2+3k$ of the involved linear forms are not annihilated by $\partial_\alpha$. Let us take a decomposition as a sum of at most $k^2+3k+2$ powers of pairwise independent linear forms. We can decompose $F$ as $F^\alpha+H$, with $F^\alpha$ encompassing all summands with base linear forms not annihilated by $\partial_\alpha$, and $H$ encompassing at most two summands, with base linear forms in $S^\alpha=\K[y,z]$. Since $G=\partial_\alpha F=\partial_\alpha F^\alpha+\partial_\alpha H=\partial_\alpha F^\alpha$, according to \cite[Prop.~3]{BuT}, $\al\left(F^\alpha\right)-\al(G)$ can not exceed the number of summands of whatever decomposition of $F^\alpha$ with all base linear forms not annihilated by $\partial_\alpha$. Since $F^\alpha$ has such a decomposition with at most $k^2+3k+2$ summands, it follows that $\al\left(F^\alpha\right)\le\al(G)+k^2+3k+2=2k^2+6k+2$. We shall find a contradiction by showing that it must also be $\al\left(F^\alpha\right)>2k^2+6k+2$. To this end, we make a direct computation of the Hilbert function of~$A^{F^\alpha}$.

Since $\partial_{\alpha^2}F^\alpha=0$, the space of all order $n$ partial derivatives of $F^\alpha$ is the sum of the space $V_{n-1}$ all order $n-1$ partial derivatives of $G$ (with $V_{-1}:=\{0\}$ by convention) and the space $W_n$ of all order $n$ partial derivatives of $F^\alpha$ with respect to $y,z$. The space $V_{n-1}$ has already been computed before for $0\le n-1\le k+2$.

Let us first consider the range $0\le n\le k-1$. Since $F^\alpha-xy^{k-1}z^{k+2}=y^{2k}z^2-H\in\K[y,z]=S^\alpha$ and
\[
xy^{k-1}z^{k+2-n}\quad,\quad\ldots\quad,\quad xy^{k-1-n}z^{k+2}
\]
are linearly independent modulo $S^\alpha$, we have $V_{n-1}\cap W_n=\{0\}$ and $\dim W_n=n+1$. Hence $\dim\left(V_{n-1}+W_n\right)=2n+1$.

When $n=k$, the partial derivatives \[\partial_{\gamma^n}F^\alpha\:,\quad\partial_{\beta\gamma^{n-1}}F^\alpha\;,\quad\ldots\;,\partial_{\beta^{n-1}\gamma}F^\alpha\] are still linearly independent modulo $S^\alpha$, because so are the monomials
\[
xy^{k-1}z^2\;,\quad\ldots\;,\quad xyz^k\;,\quad xz^{k+1}\;.
\]
But now $\partial_{\beta^n}F^\alpha\in S^\alpha$. However, we can check that it still lies outside $V_{n-1}=z^3S^\alpha_{k-1}$. Suppose, on the contrary, that $\partial_{\beta^n}F^\alpha\in z^3S^\alpha_{k-1}$. Since $\partial_{\beta^n}F^\alpha+\partial_{\beta^n}H=\partial_{\beta^n}F=\partial_{\beta^k}F$ is a scalar multiple of $y^kz^2$, we would have $\partial_{\beta^n}H\in z^2S^\alpha_k$. But $H$, and hence $\partial_{\beta^k}H$, can be written as a linear combination of two powers of pairwise independent linear forms (using zero coefficients if $H$ has less than two summands). Therefore, by \autoref{Pre} we would have that $\partial_{\beta^k}H$ is a scalar multiple of $z^{k+2}\in z^3S^\alpha_{k-1}$ (taking into account the initial assumption $k\ge1$). Hence $\partial_{\beta^k}F$ would lie in $z^3S^\alpha_{k-1}$ but, on the contrary, it is a nonzero scalar multiple of $y^kz^2$. We conclude that $\dim\left(V_{n-1}+W_n\right)=2n+1$ even when $n=k$.

Up till now we have shown that the Hilbert function of $A^{F^\alpha}$ takes value $2n+1$ at $n$ and at $2k+2-n$ for $0\le n\le k$, and we are checking that $\al\left(F^\alpha\right)>2k^2+6k+2$. Hence we can conclude the proof by showing that the value at $n=k+1$ is strictly greater than
\[
2k^2+6k+2-2\sum_{n=0}^{k}(2n+1)=2k^2+6k+2-2(k+1)^2=2k\;.
\]
Like before, let us first note that since
\[
xy^{k-1}z\;,\quad\ldots\;,\quad xyz^{k-1}\;,\quad xz^k
\]
are linearly independent modulo $S^\alpha$, so are the partial derivatives $\partial_{\gamma^n}F^\alpha$, $\ldots$, $\partial_{\beta^{n-3}\gamma^3}F^\alpha$, $\partial_{\beta^{n-2}\gamma^2}F^\alpha$. Since $\dim V_{n-1}=k$, this shows that the value of the Hilbert function to be computed is at least $2k$, and that we need only one more independent partial derivative. It will be enough to check that at least one of $\partial_{\beta^n}F^\alpha$ and $\partial_{\beta^{n-1}\gamma}F^\alpha$, which clearly are both in $S^\alpha$, is not in $V_{n-1}=z^2S^\alpha_{k-1}$. If $\partial_{\beta^n}F^\alpha\in V_{n-1}$, since $\partial_{\beta^n}F^\alpha+\partial_{\beta^n}H=\partial_{\beta^n}F=\partial_{\beta^{k+1}}F\in z^2S^\alpha_{k-1}=V_{n-1}$, by \autoref{Pre}  we have that $\partial_{\beta^n}H$ is a scalar multiple of $z^{k+1}$. But in this case we have $\partial_{\beta^{n-1}}H=0$ ($n-1=k\ge 1$), hence $\partial_{\beta^{n-1}\gamma}F^\alpha\not\in z^2S^\alpha_{k-1}=V_{n-1}$ because $\partial_{\beta^{n-1}\gamma}F^\alpha=\partial_{\beta^{n-1}\gamma}F^\alpha+\partial_{\beta^{n-1}\gamma}H=\partial_{\beta^{n-1}\gamma}F=\partial_{\beta^k\gamma}F$ is a nonzero scalar multiple of $y^kz$.
\end{proof}

Let $\rmax(n,d)$ be the maximum Waring rank of forms of degree $d$ in $n$ variables. For $n=3$ and $d\ge 2$,  \autoref{Main} and \cite[Th.~1]{BuT} (which holds for $\K$ as well as for $\mathbb{C}$ with no modifications in the proof) give the following lower bound.

\begin{thm}
For $d\ge 2$ we have \[\rmax(3,d)\ge\left\lfloor\frac{d^2+2d+5}4\right\rfloor\;.\]
\end{thm}

At the time of writing, we are inclined to believe that the opposite inequality has some chance of being true as well.

\end{document}